\documentclass[12pt]{amsart}
\usepackage{amssymb}
\usepackage{amsfonts}
\usepackage{amscd}
\usepackage[T1]{fontenc}

\swapnumbers

\newcommand{\Q}{\mathbb{Q}}
\newcommand{\F}{\mathbb{F}}
\newcommand{\Z}{\mathbb{Z}}
\newcommand{\N}{\mathbb{N}}

\def\11{{\mathbf 1}}

\def\cM{{\mathcal M}}

\def\cO{{\mathcal O}}

\mathchardef\alphag="7C0B
\mathchardef\betag="7C0C
\mathchardef\gammag="7C0D
\mathchardef\deltag="7C0E
\mathchardef\varepsilong="7C22
\mathchardef\varphig="7C27
\mathchardef\psig="7C20
\mathchardef\zetag="7C10
\mathchardef\epsilong="7C0F
\mathchardef\rhog="7C1A
\mathchardef\taug="7C1C
\mathchardef\upsilong="7C1D
\mathchardef\iotag="7C13
\mathchardef\thetag="7C12
\mathchardef\pig="7C19
\mathchardef\sigmag="7C1B
\mathchardef\etag="7C11
\mathchardef\omegag="7C21
\mathchardef\kappag="7C14
\mathchardef\lambdag="7C15
\mathchardef\mug="7C16
\mathchardef\xig="7C18
\mathchardef\chig="7C1F
\mathchardef\nug="7C17
\mathchardef\varthetag="7C23
\mathchardef\varpig="7C24
\mathchardef\varrhog="7C25
\mathchardef\varsigmag="7C26
\mathchardef\Omegag="7C0A
\mathchardef\Thetag="7C02
\mathchardef\Sigmag="7C06
\mathchardef\Deltag="7C01
\mathchardef\Phig="7C08
\mathchardef\Gammag="7C00
\mathchardef\Psig="7C09
\mathchardef\Lambdag="7C03
\mathchardef\Xig="7C04
\mathchardef\Pig="7C05
\mathchardef\Upsilong="7C07

\newtheorem{thm}{Theorem}
\newtheorem{lem}{Lemma}
\newtheorem{cor}[thm]{Corollary}

\newtheorem{claim}{Claim}

\newtheorem*{cor*}{Corollary}

\theoremstyle{remark}

\theoremstyle{plain}

\numberwithin{equation}{subsection}

\def\boxit#1#2{\setbox1=\hbox{\kern#1{#2}\kern#1}%
\dimen1=\ht1 \advance\dimen1 by #1
\dimen2=\dp1 \advance\dimen2 by #1
\setbox1=\hbox{\vrule height\dimen1 depth\dimen2\box1\vrule}%
\setbox1=\vbox{\hrule\box1\hrule}%
\advance\dimen1 by .4pt \ht1=\dimen1
\advance\dimen2 by .4pt \dp1=\dimen2 \box1\relax}

  {\par\medskip\noindent #1\par\begingroup%
    \advance\leftskip by 1em\advance\rightskip by 1em}%
  {\par\endgroup}

\begin{document}

\title[Model Completeness for Henselian Fields]{Model Completeness for Henselian fields with finite ramification valued in a $Z$-group}
\author{Jamshid Derakhshan, Angus Macintyre${}^\dag$}

\thanks{${}^\dag$Supported by a Leverhulme Emeritus Fellowship}

\address{St.~ Hilda's College, Cowley Place, Oxford OX4 1DY and Mathematical Institute, University of Oxford, 
Andrew Wiles Building, Woodstock Road, Oxford, OX2 6GG, UK}
\email{derakhsh@maths.ox.ac.uk}

\address{Queen Mary, University of London,
School of Mathematical Sciences, Queen Mary, University of London, Mile End Road, London E1 4NS, UK}
\email{angus@eecs.qmul.ac.uk}

\subjclass[2000]{}

\begin{abstract} We prove that the theory of a Henselian valued field of characteristic zero, with finite ramification, and 
whose value group is a $\Z$-group,  
is model-complete in the language of rings if the theory of its residue field is model-complete in the language of rings. We apply this to prove that every infinite algebraic extension of the field of $p$-adic numbers $\Q_p$ 
with finite ramification is model-complete in the language of rings. For this, we give a necessary and sufficient condition for model-completeness of the theory of a perfect pseudo-algebraically closed field with pro-cyclic absolute 
Galois group.
\end{abstract}

\maketitle

\section{Introduction}

Model completeness for the theory of $p$-adic numbers $\Q_p$ in the language of rings 
follows from the theorem of Macintyre \cite{Macintyre} on quantifier elimination for $\Q_p$ in the Macintyre language.  
For a finite extension of $\Q_p$, model-completeness in the ring language can be deduced from a theorem of Prestel and 
Roquette \cite[Theorem 5.1, pp. 86]{PR} on model-completeness in the 
Prestel-Roquette language combined with an 
existential definition of the valuation ring in the ring language due to B\'{e}lair \cite{belair}. The language of 
Prestel-Roquette involves certain 
constant symbols and a symbol for the valuation. However, 
the proof of the theorem of Prestel-Roquette gives model completeness in the language of rings. Model-completeness for a finite extension of $\Q_p$ in the language of rings has also been deduced in \cite{DM-pres} from model-completeness for the groups of multiplicative residue classes of the field which follows from model-completeness for certain 
pre-ordered abelian groups which we call finite-by-Preburger groups. (Curiously, the model-completeness in the ring language for a finite extension of $\Q_p$ was not observed till \cite{DM-pres}). In this paper, we shall prove a general model-completeness result for Henselian fields, and apply it to certain infinite 
extensions of $\Q_p$.

Let us recall that an ordered abelian group is called a $\Z$-group if it is 
elementarily equivalent to $\Z$ as an ordered abelian group. Let $K$ be a valued field with valuation $v$ and 
residue field $k$. The (absolute) ramification index $e$ of $K$ is defined to be the cardinality of the set if 
elements $\gamma$ such that $0<\gamma\leq v(p)$ if $k$ has characteristic $p>0$, and defined to be $0$ if $k$ has characteristic $0$. 
If $e<\infty$, we say that $K$ is finitely ramified or has finite ramification $e$. If $e=0$ or $e=1$, we say that $K$ is 
unramified. For example, the extension of $\Q_p$ got by adjoining an $e$th root of $p$, has ramification index $e$, whereas an extension of $\Q_p$ got by adjoining roots of unity of order prime to $p$ is unramified. 

Our main result is the following.
\begin{thm}\label{mc-hens} Let  $K$ be a Henselian valued field of characteristic zero with finite ramification. Suppose the value group of $K$ is a $\Z$-group.  
If the theory of the residue field of $K$ is model-complete in the language of rings, then the theory of $K$ is 
model-complete in the language of rings.\end{thm}

In the case when the residue field of $K$ has characteristic zero, the result is a well-known consequence of the Ax-Kochen-Ershov theory and was worked out in detail by Ziegler in \cite{ziegler-thesis}. Thus we shall assume that the residue field $k$ has characteristic $p>0$. 


\begin{thm}\label{p-adic} Let $K$ be an infinite algebraic extension of $\Q_p$ with finite ramification. Then the theory of $K$ in the language of rings is model-complete.\end{thm}

(The case of {\it finite} extensions of $\Q_p$ is discussed earlier). Let us recall that a field $K$ is called pseudo-algebraically closed if every absolutely irreducible variety defined over $K$ has a $K$-rational point. 
To deduce Theorem  \ref{p-adic} from Theorem \ref{mc-hens} we prove the following result which gives a necessary and sufficient condition for model-completeness in the ring language of the theory of perfect pseudo-algebraically closed fields with pro-cyclic absolute Galois group. Given a field $K$, we denote the absolute Galois group of $K$ by $Gal(K)$.
\begin{thm}\label{mc} Let $K$ be a perfect pseudo-algebraically closed field such that $Gal(K)$ is pro-cyclic. 
Let $k$ denote the prime subfield of $K$. Then the theory of $K$ in the language of rings is model-complete if and only if 
\begin{equation}\label{cond}
K^{alg}=K \otimes_{Abs(K)} k^{alg},\end{equation}
that is, every finite algebraic extension of $K$ is generated by elements that are algebraic over $k$.
\end{thm}
To prove Theorem \ref{mc}, we use the elementary invariants given by Cherlin-van den Dries-Macintyre \cite{CDMM} for the theory of pseudo-algebraically closed fields (using a model theory for $G(K)$ dual to that of $K$) generalizing Ax's work for the pseudo-finite case \cite{ax}.  

By the Lang-Weil estimates or the theorem of Andr\'{e} Weil on the Riemann hypothesis for curves over finite fields, any infinite algebraic extension $K$ of $\F_p$ is pseudo-algebraically closed (see \cite[Corollary 11.2.4]{FJ} for details). For such a $K$, $Gal(K)$ is pro-cyclic (see \cite[chapter 1]{FJ}). Thus Theorem \ref{p-adic} follows from Theorem \ref{mc} and Theorem \ref{mc-hens}.

\section{Proof of Theorem \ref{mc-hens}}

Let $K$ be a valued field. We shall denote the valuation on $K$ by $v_K$ or $v$, the ring of integers 
of $K$ by $\cO_K$, the valuation ideal by $\cM_K$, and the value group by $\Gamma_K$ or $\Gamma$. We 
denote the residue field by $k$.

Assume throughout that $K$ has characteristic zero and residue characteristic $p>0$. 
We take the smallest 
convex subgroup $\Delta$ of $\Gamma_K$ containing 
$v(p)$ and consider the quotient $\Gamma_K/\Delta$ with the ordering coming from convexity of $\Delta$ (see \cite{schilling}). $K$ carries a valuation which is the composition of $v_K$ with the canonical surjection $\Gamma_K \rightarrow \Gamma_K/\Delta$. This valuation will be denoted by $\dot{v}: K\rightarrow 
\Gamma_K/\Delta\cup \{\infty\}$ and is called the {\it coarse valuation corresponding to $v$}. We denote the valued field $(K,\dot{v})$ by $\dot{K}$. The valuation ring of $\cO_{\dot{K}}$ of $\dot{v}$ is the set $\{x\in K: \exists \delta \in \Delta~(v(x)\geq \delta)\}$. It is also the smallest overring of $\cO_K$ in which $p$ becomes a unit, or the localization of $\cO_K$ with respect to the multiplicatively closed set $\{p^m: m\in \N\}$. The maximal ideal $\cM_{\dot{K}}$ of $\dot{v}$ is the set $\{x\in K: \forall \delta~ (v(x)>\delta)\}$. Clearly 
 $\cM_{\dot{K}}\subseteq \cM_K$. The residue field of $K$ with respect to the coarse valuation $\dot{v}$ 
has characteristic zero, and is called the {\it core field} of $K$ corresponding to $v$. It is denoted by $K^{\circ}$. The core field carries a valuation $v_0$ defined by $v_0(x+\cM_{\dot{K}})=v(x)$. The valuation $v_0$ has value group $\Delta$, valuation ring $\cO_K/\cM_{\dot{K}}$, maximal ideal  $\cM_K/\cM_{\dot{K}}$, and residue field $k$. The residue degree of $K$ is defined to be the dimension over $\F_p$ of the residue field $k$.


\begin{lem}\label{index-lem} The ramification index and residue degree of $K$ and the core field $K^{\circ}$ 
are the same.\end{lem}
\begin{proof} For a proof see \cite[pp. 27]{PR}.
\end{proof}

\

We recall that a sequence $\{a_n\}_{n\in \omega}$ of elements of a valued field is called $\omega$-{\it pseudo-convergent} if for some integer $n_0$, we have $v(a_m-a_n)>v(a_n-a_k)$ for all $m>n>k>n_0$. An element $a\in K$ is called a pseudo-limit of the sequence $\{a_n\}$ if for some integer $n_0$ we have $v(a-a_n)>v(a-a_k)$ for all $n>k>n_0$. The field $K$ is called $\omega$-pseudo-complete if every $\omega$-pseudo-convergent sequence of length $\omega$ has a pseudo-limit in the field. We shall use the following lemma.
\begin{lem}\label{psc-lem} An $\aleph_1$-saturated valued field is $\omega$-pseudo-complete.\end{lem}
\begin{proof}
Obvious.
\end{proof}

\

We shall need the following result on existential definability of valuation rings.

\begin{lem}\label{bel} Let $K$ be a Henselian valued field of characteristic zero, residue characteristic $p>0$, and ramification index  $e>0$. Let $n>e$ be an integer that is  not 
divisible by $p$. Then the valuation ring $\cO_K$ is existentially definable by the formula $\exists y~(1+px^n=y^n)$.\end{lem}
\begin{proof} This is proved in \cite[Lemma 1.5,pp. 4]{belair} under the assumption of a finite residue field but the same proof goes through in the more general case as follows. Let  $x\in \cO_K$. Let $f(y):=y^n-px^n -1$. 
Then $v(f(1))>2v(f'(1))$, so $f$ has a root in $K$ by Hensel's Lemma. Conversely,  suppose $1+px^n$ is an $n$th power. If $v(x)<0$, then $v(px^n)<0$, and so $v(y)<0$, hence $nv(y)=e+nv(x)$, thus $n$ divides $e$, 
contradiction to the choice of $n$.\end{proof}

{\bf Note} The existential definition above is uniform once one fixes $p$ and a finite bound on the ramification index $e$. 
In particular, for any extension $K$ of $\Q_p$ with ramification index $e$, the valuation ring of $K$ is defined 
by an existential formula of the language of rings that depends only on $p$ and $e$, and not $K$.

\begin{cor}\label{bel-cor} Suppose that $K_1 \subseteq K_2$ is an extension of Henselian valued fields of 
characteristic $0$ and residue characteristic $p>0$, and whose value groups are $\Z$-groups. Suppose that the index of ramification of $K_1$ and $K_2$ is $e$ where $0<e<\infty$. Then  
$$\cO_{K_2} \cap K_1=\cO_{K_1}.$$
 \end{cor}
\begin{proof} First note that given a valued field $K$ of residue characteristic $p>0$ whose value group is a $\Z$-group and 
which has ramification index $e$, we have that
$$\cM_{K}=\{x\in K: x^e p^{-1} \in \cO_K\}.$$
Indeed, suppose that $x\in \cM_K$. Then $ev(x)-e\geq 0$, thus since $v(p)=e$, we deduce that 
$x^ep^{-1}\in \cO_K$. Conversely, suppose that $x\in K$ satisfies the condition $x^e p^{-1} \in \cO_K$. Then 
$ev(x)-e\geq 0$, hence $ev(x)\geq e$, so $v(x)\geq 1$. 

From this observation and the existential definability of $\cO_{K_1}$ and $\cO_{K_2}$ by the same formula given 
by Lemma \ref{bel}, we deduce that the maximal ideals $\cM_{K_1}$ and $\cM_{K_2}$ are definable by the same 
existential formula (of the language of rings). 

Now we can complete the proof of the Corollary. From the existential definability of $\cO_{K_1}$ and $\cO_{K_2}$ by the 
same formula we deduce that $\cO_{K_1} \subseteq \cO_{K_2} \cap K_1$. For the other direction, suppose that 
there is an element $\beta \in K_1 \cap \cO_{K_2}$ but $\beta \notin \cO_{K_1}$. Then $\beta^{-1} \in \cO_{K_1}$, 
hence $\beta^{-1} \in \cO_{K_2}$. Thus $\beta$ is a unit in $\cO_{K_2}$. From $\beta \notin \cO_{K_1}$ we deduce that 
$\beta^{-1} \in \cM_{K_1}$, so $\beta^{-1} \in \cM_{K_2}$, contradiction. This proves the corollary. 
\end{proof}

We can now give the proof of Theorem \ref{mc-hens}. Let $K_1 \subseteq  K_2$ be an embedding of models of $Th(K)$. By Corollary \ref{bel-cor}, this is an embedding of valued fields. Thus there is a natural inclusion of the residue field  (resp.~ value group)  of $K_1$ into the residue field (resp.~value group) of $K_2$. We make a series of reductions.

\

{\bf Step 1}

We may assume that $K_1$ and $K_2$ are $\aleph_1$-saturated. 
Indeed, we can form ultrapowers $K_1^{U}$ and $K_2^{U}$ of $K_1$ and $K_2$, for a non-principal ultrafilter $U$. 
If we know that $K_1^U$ is an elementary substructure of $K_2^U$, then since $K_i$ is an elementary substructure of 
$K_i^U$ for $i=1,2$, we deduce that $K_1$ is an elementary substructure of $K_2$. 

\

{\bf  Step 2} 

It suffices to prove that the core field $K_1^{\circ}$ is an elementary substructure of the core field $K_2^{\circ}$. To see this, note that since the coarse valued fields  $\dot{K_1}$ and $\dot{K_2}$ have characteristic zero residue fields $K_1^{\circ}$ and $K_2^{\circ}$ respectively, and divisible torsion-free abelian value groups, and the theory of divisible torsion-free abelian groups is 
model-complete, by the work of Ax-Kochen-Ershov as spelled out by Ziegler \cite{ziegler-thesis}, we deduce that the embedding of $K_1$ in $K_2$ is elementary provided the embedding of $K_1^{\circ}$ into $K_2^{\circ}$ is elementary.

\

{\bf Step 3} 

We prove the embedding of $K_1^{\circ}$ into $K_2^{\circ}$ is elementary. Since the fields $K_1$ and $K_2$ are $\aleph_1$-saturated, by Lemma \ref{psc-lem} they are $\omega$-pseudo-complete. Thus the valued fields $K_1^{\circ}$ and $K_2^{\circ}$ are also $\omega$-pseudo-complete (since the map $\Gamma \rightarrow \Gamma/\Delta$ is order-preserving). However, these fields are valued in $\Delta$ which is canonically isomorphic to $\Z$. 
Thus $K_1^{\circ}$ and $K_2^{\circ}$ are Cauchy complete. By Lemma \ref{index-lem}, the ramification index of $K_1^{\circ}$ and $K_2^{\circ}$ is the same as the ramification 
index of $K_1$ and $K_2$ which equals the ramification index of $K$ which is $e$. 

By the structure theorem for complete fields with ramification index $e$ 
(see \cite[Theorem 4,pp.37]{serre-local}), 
$K_1^{\circ}$ and $K_2^{\circ}$ are respectively finite extensions of degree $e$, obtained by adjoining a uniformizing element, of the fields $W(k_1)$ and $W(k_2)$ which are fraction fields of the rings of Witt vectors of $k_1$ and $k_2$ respectively, where $k_1$ and $k_2$ are the residue fields of $K_1^{\circ}$ and $K_2^{\circ}$ (which coincide with the residue fields of $(K_1,v_{K_1})$ and $(K_2,v_{K_2})$ respectively).

Thus $K_1^{\circ}=W(k_1)(\pi)$ for some uniformizing element $\pi \in K_1^{\circ}$. 
$\pi$ is the root of a polynomial 
$$E(x):=x^e+c_{x}^{e-1}+\dots+c_e$$
that is Eisenstein over $W(k_1)$. So  
$$c_j\in \cM_{W(k_1)}$$ for all $j$ and 
$$c_e\in \cM_{W(k_1)}-\cM_{W(k_1)}^2.$$

\begin{claim} $E(x)$ is Eisenstein over $W(k_2)$ and $K_2^{\circ}=W(k_2)(\pi)$.
\end{claim}
\begin{proof}  
The condition that $c_j$ is in the maximal ideal $\cM_{W(k_1)}$ is equivalent to the condition that 
$$c_j^ep^{-1}\in \cO_{W(k_1)}$$ 
since this condition means that $ev(c_j)-e\geq 0$, that is $ev(c_j)\geq e$, which is $v(c_j)\geq 1$; and the condition that $c_e$ is a uniformizer, i.e. that it lies in 
$\cM_{W(k_1)}$ and does not lie in $\cM_{W(k_1)}^2$, is equivalent to the conjunction of the 
statements $c_e^ep^{-1} \in \cO_{W(k_1)}$ and $c_e^{-e}p\in \cO_{W(k_1)}$. Indeed, the latter condition is equivalent to $-ev(c_e)+e\geq 0$, i.e. $-ev(c_e)\geq -e$, i.e., $v(c_e)\leq 1$.

By Lemma \ref{bel}, the valuations on $W(k_1)$ and $W(k_2)$ are existentially definable by the same formula since these fields are absolutely unramified and $p$ is 
a uniformizer in both. We deduce from the preceding argument that 
$$c_j\in \cO_{W(k_2)}$$
for all $j$ and 
$$c_e\in \cM_{W(k_2)}-\cM_{W(k_2)}^2.$$
Therefore $E(x)$ is an Eisenstein polynomial over $W(k_2)$, and $\pi$ remains a uniformizer in $K_2^{\circ}$. 

Now $K_1^{\circ}$ is an extension of degree $e$ of $W(k_1)$ generated by $\pi$. 
As $\pi$ remains a uniformizer in $K_2^{\circ}$ and $E(x)$ is Eisenstein over $W(k_2)$, the extension 
$W(k_2)(\pi)$ has degree $e$ over in $W(k_2)$. But $\pi \in K_2^{\circ}$ and $K_2^{\circ}$ has degree $e$ over $W(k_2)$ as well, so we deduce that
$$K_2^{\circ}=W(k_2)(\pi).$$
\end{proof}

\begin{claim} The embedding of $W(k_1)$ in $W(k_2)$ is elementary.\end{claim}
\begin{proof} Since $k_1$ and $k_2$ are residue fields of $K_1$ an $K_2$ for the valuation $v_K$, 
the embedding of $k_1$ in $k_2$ is elementary. Since $K_1$ and $K_2$ are $\aleph_1$-saturated, the fields $k_1$ and $k_2$ are 
also $\aleph_1$-saturated. Given any finitely many elements $a_1,\dots,a_m$ from $W(k_1)$, there is an 
isomorphism from $W(k_1)$ to $W(k_2)$ fixing $a_1,\dots,a_m$ since elements of $W(k_1)$ and $W(k_2)$ 
can be represented in the form $\sum_i c_i p^i$, where $c_i$ are from the residue field. The countable subfields of $k_1$ and $k_2$ form a back-and-forth system. This induces a back-and-forth system between $W(k_1)$ and $W(k_2)$, and it follows that the embedding of $W(k_1)$ into $W(k_2)$ is elementary\end{proof}

It remains to prove that the embedding of $K_1^{\circ}$ into $K_2^{\circ}$ is elementary. We interpret $W(k_i)(\pi)$ inside $W(k_i)$ (for $i=1,2$) in the usual way as follows. We identify 
$W(k_i)(\pi)$ with $W(k_i)^e$. On the $e$-tuples we define addition as the usual addition on vector spaces and multiplication by 
$$(x_1,\dots,x_e)\times (y_1,\dots,y_e)=(x_1I_e+x_2M_{\pi}+\dots+x_eM_{\pi}^{e-1}) \left( \begin{array}{c}
                                                                                                                                                                          y_1\\
                                                                                                                                                                           y_2\\
                                                                                                                                                                           \vdots\\
                                                                                                                                                                            y_e \end{array}\right)$$
where $I_e$ is the identity $e \times e$-matrix and $M_{\pi}$ is the $e \times e$-matrix of multiplication by $\pi$. Note that $M_{\pi}$ depends uniformly only on the coefficients $c_0,\dots,c_{e-1}$ of $E(x)$. Using Claim 2, we deduce that 
the embedding $W(k_1)(\pi) \rightarrow W(k_2)(\pi)$ is elementary. Thus $K_1^{\circ} \rightarrow K_2^{\circ}$ is elementary. The proof of Theorem \ref{mc-hens} is complete. 

\section{Model completeness for pseudo algebraically closed fields and proof of Theorem \ref{mc}}

Given a field $K$, the field of absolute numbers of $K$ is defined by $Abs(K):=k^{alg}\cap K$, where $k$ is the prime subfield of $K$. By a result of Ax \cite{ax}, two perfect pseudo-algebraically closed fields $K_1$ and $K_2$ 
whose absolute Galois groups are isomorphic to $\hat{\Z}$ are elementarily equivalent if and only if $Abs(K_1)=Abs(K_2)$. In other words, the theory of a such a field is determined by its absolute numbers $Abs(K)$ (equivalently by 
the polynomials $f \in k[x]$ that are solvable in $K$). 

Elementary invariants for pseudo-algebraically closed fields were given by Cherlin-van den Dries-Macintyre in \cite{CDMM},\cite{CDMM2} in terms of the language for profinite groups. In this case one has to preserve the degree of imperfection and 
the co-elementary theory defined as follows. The language CSIS for complete stratified inverse systems is a language with 
infinitely many sorts indexed by $\N$, each sort is equipped with the group operation. The $n$th sort describes properties 
for the set of groups in the inverse system which have cardinality $n$.  The language has in addition symbols for the connecting 
canonical maps between the groups in different sorts. Given any profinite group $G$, the set of 
finite quotients of $G$ with the canonical maps between them is a stratified inverse system. A coformula is a formula
of the language CSIS. A profinite group cosatisfies a cosentence if the associated stratified inverse system satisfies 
the cosentence. A cosentence or coformula has a translation to the language of fields. For details see \cite{CDMM}.

For any field $K$, the Galois diagram of $K$ is defined to be the theory
$$\{\exists \bar x,\bar y,\bar z,\bar t ~ (\varphi(\bar x,\bar y,\bar z,\bar t) \wedge \delta(\bar x,\bar y,\bar z,\bar t)):$$
$$\exists \bar a,\bar b,\bar c,\bar d \in Abs(K) ~ (K\models \varphi'(\bar a,\bar b,\bar c) \wedge 
\delta(\bar a,\bar b,\bar c,\bar d))\}$$
where $\delta(\bar x,\bar y, \bar z, \bar t)$ describes the isomorphism type of the field generated by $\bar x,\bar y, \bar z, \bar t$, 
and $\varphi$ is a coformula and $\varphi'$ its "translation" into the language of rings (cf. \cite{CDMM}).

We then have the following result. 
\begin{thm}\cite{CDMM}\label{cdm} Two pseudo-algebraically closed fields $K$ and $L$ are elementarily equivalent if and only if  $K$ and $L$ have the same characteristic and same degree of imperfection, and $\Delta(K)=\Delta(K)$.\end{thm}

We also need the following results.
\begin{thm}\cite{FJ}\label{hilb} Infinite finitely generated fields are Hilbertian.\end{thm}
\begin{proof} See \cite{FJ}, Theorem 13.4.2, pp. 242.\end{proof}

\begin{thm}(Jarden) \label{jarden-pac}If $L$ is a countable Hilbertian field, then the set 
$$\{\sigma \in Gal(L): Fix(\sigma) ~ \rm{is~ pseudofinite}\}$$
has measure $1$.
\end{thm}
\begin{proof} See \cite[pp.76]{jarden-pac} or \cite[Theorem 18.6.1,pp. 380]{FJ}.\end{proof}

Now we can give the proof of Theorem \ref{mc}. The condition \ref{cond} implies that every finite algebraic extension $K(\alpha)$ of $K$ is generated by elements algebraic 
over $k$, and thus by the primitive element theorem, by a single algebraic element $\alpha$. 

Now all this is part of the theory $Th(K)$. For example, the unique extension of $K$ of dimension $n$ is generated by a root 
$\alpha$ of some polynomial $f$ over $k$. Fix the minimum polynomial $f$ of $\alpha$. Then we just say that some root of $f$ 
generates the unique extension of $K$ of dimension $n$. This will be true for any $L$ with $L \equiv K$. It follows that any 
embedding $L\rightarrow K_1$ of models of $Th(K)$ is regular, and thus elementary (cf. Cherlin-van den Dries-Macintyre \cite{CDMM} or Jarden-Kiehne \cite{JK}).

Conversely, Suppose that \ref{cond} does not hold. We shall prove that $Th(K)$ is not model-complete. Since 
$K^{alg} \neq K \otimes_{Abs(K)} k^{alg}$, there is some finite algebraic extension $K(\alpha)$ that is not included in 
any $K(\beta)$, where $\beta$ is algebraic over $k$. Now consider such a field $K(\alpha)$ of minimal dimension $d$ over $K$.  $K(\alpha)$ 
is normal cyclic over $K$, so $d$ is a prime $p$, otherwise, $d=p_1^{k_1}\dots p_r^{k_r}$, where $n>1$, and each of the degree $p_j^{k_j}$ extensions is included in some 
$K(\beta)$ that is algebraic over $k$, and so $K(\alpha)$ is too, contradiction. Thus $K(\alpha)$ is a dimension $p$ extension of $K$. 









Now let $f$ be the minimum polynomial of $\alpha$ over $K$. Put 

$$\Lambda:=Diag(K) \cup \Sigma_{PAC} \cup \{\exists x~ (f(x)=0)\} \cup \{ \forall x~ (g(x) \neq 0), g\in \Theta\} \cup \Delta(K)$$
were $\Sigma_{PAC}$ denotes the set of sentences expressing the condition of being pseudo-algebraically closed, $\Theta$ is the set of polynomials in one variable over $k$ which are unsolvable in $K$, and $\Delta(K)$ is the Galois diagram of $K$.

\begin{claim} $\Lambda$ is consistent.\end{claim}

We do a compactness argument. 
Consider a finite subset $\Lambda_0$ of $\Lambda$. It involves a finite set $c_0,\dots,c_m$ from $K$ including coefficients of $f$ and a finite part of $Diag(K)$, finitely many 
$g_1,\dots,g_l$ from $\Theta$, a $t$ with $f(t)=0$, and a finite part of the Galois diagram $\Delta(K)$. Given a finite part $S$ of 
the Galois diagram $\Delta(K)$, $S$ contains finitely many statements describing the isomorphism types of fields generated by finitely many finite subsets $S_1,\dots,S_k$ of $K$, and translations to the language of rings of finitely many 
coformulas. The translations of the coformulas involve Galois groups of finitely many finite extensions of $K$. The compositum of these is a finite Galois extension $K(T)$ of $K$, for a finite set $T$.


Note that $tr.deg.(k(\alpha,c_0,\dots,c_m,S_0,\dots,S_k,T))\geq 1$, so  by Theorem \ref{hilb}, 
$k':=k(\alpha,c_0,\dots,c_m,S_0,\dots,S_k,T)$ is Hilbertian. Note that $f$ is irreducible of dimension $p$ over $k'$. Now if 
we adjoin to $k'$ a root $\alpha$ of $f$, then none of $g_1,\dots,g_l$ get a root. For if one does, that root is either in 
$k'$ which is impossible, or has dimension congruent to zero modulo $p$ over $k'$, and 
then $\alpha \in K(\beta)$, for some $\beta$ which is algebraic over $k$.

So now apply Theorem \ref{jarden-pac} to $k'$ and deduce that the set of 
all $\sigma \in G(k')$ such that  $Fix(\sigma)$ is pseudofinite has measure $1$. Note that given a polynomial 
$g(x)$ over $k$, the set $$G_g:=\{\sigma \in G(k'): Fix(\sigma)~\text{does not contain a root of $g$}\}$$
is open in $G(k')$ since $U:=Gal(k'^{alg}/F)$ is a basic open set containing the identity in $G(k')$ where $F$ is the splitting field of $g(x)$, and $\sigma U \in G_g$ for any $\sigma\in G_g$. Thus the set 
$$\{\sigma \in G(k'): \text{$g_1,\dots,g_l$ do not have a root in $Fix(\sigma)$ and $Fix(\sigma)$ is pseudofinite}\}$$
has non-zero measure. Note that for any such $\sigma$, the fixed field $Fix(\sigma)$ contains the given finite part of $Diag(K)$, contains a root of $f$ (namely $\alpha$), and contains 
$T$. Thus $Fix(\sigma)$ must satisfy the finitely many given statements from the Galois diagram $\Delta(K)$ and the diagram $Diag(K)$ as these can be witnessed by finitely many elements from 
$S_1\cup \dots \cup S_k \cup T$ (by adding constants symbols). We deduce that $Fix(\sigma)$ is a model of $\Lambda_0$. Thus $\Lambda$ has a model $L$.


We need to show that $\Delta(K)=\Delta(L)$. It is obvious that $\Delta(K) \subseteq \Delta(L)$. We show that $\Delta(L) \subseteq \Delta(K)$. Suppose that $\psi \in \Delta(L)$ and $\psi \notin \Delta(K)$. 
Then $\psi$ involves statements on isomorphism type and translations of coformulas corresponding to a finite subset of $Abs(L)$ that does not hold for $K$. But $Abs(K)=Abs(L)$, so $\neg \psi$ holds for the finitely many 
elements of $Abs(K)$, hence $\neg\psi \in\Delta(L)$ by adding constants for the distinguished elements of $Abs(K)=Abs(L)$, which is a contradiction. 

Applying Theorem \ref{cdm}, we deduce that $K$ and $L$ are elementarily equivalent. Clearly $K$ is not an elementary submodel of $L$. This completes the proof.

\



\bibliographystyle{amsplain}
\bibliography{anbib}
\end{document}